\documentclass[11pt, one side,emlines]{amsart}
\usepackage{amssymb,latexsym,xy,eucal,mathrsfs}
\usepackage{MnSymbol}
\usepackage{amssymb,latexsym,xy,eucal,mathrsfs, graphicx, tikz, amsmath, caption}
\textwidth=15.2cm \textheight=24cm \theoremstyle{plain}
\newtheorem{lemma}{Lemma}[section]
\newtheorem{theorem}[lemma]{Theorem}

\newtheorem{corollary}[lemma]{Corollary}

\theoremstyle{definition}
\newtheorem{example}[lemma]{Example}
\newtheorem{remark}[lemma]{Remark}

\numberwithin{equation}{section}  \voffset
-55truept \hoffset -15truept
\begin{document}
\baselineskip 15truept
\title{Generalized Projections in $\mathbb Z_n$}
\date{}
\author{ Anil Khairnar and B. N. Waphare}
\address{\rm Department of Mathematics, Abasaheb Garware College, Pune-411004, India.}
 \email{\emph{anil.khairnar@mesagc.org; anil\_maths2004@yahoo.com}}
\address{\rm Center for Advanced Studies in Mathematics, Department of Mathematics, 
Savitribai Phule Pune University, Pune-411007, India.}
 \email{\emph{bnwaph@math.unipune.ac.in; waphare@yahoo.com}}
 \subjclass[2010]{Primary 11A25; Secondary 06A06} 
\maketitle {\bf \small Abstract:} {\small  
 We consider the ring $\mathbb Z_n$ (integers modulo $n$) with the partial order `$\leq$' given by `$a \leq b$ if either $a=b$ or $a\equiv ab~(mod~n)$'. In this paper, we obtain necessary and sufficient conditions for the poset ($\mathbb Z_n,~\leq$) to be a lattice.  
 }\\
 \noindent {\bf Keywords:} generalized projections, regular elements, nilpotent elements.
\section{Introduction} 
\indent An element $a$ in a commutative ring $R$ is said to be a {\it generalized projection} 
if $a^k=a$ for some $k \in \mathbb N$ with $k\geq 2$ (see \cite{Anil}); an element $a$ is called a {\it regular element} if $a=aba(=a^2b)$ for some element $b$ in the ring. It is proved by L$\acute{\text a}$szl$\acute{\text o}$ T$\acute{\text o}$th \cite{Las} that in $\mathbb Z_n$, an element is a generalized projection if and only if it is regular; in-fact the following result is proved. 
\begin{theorem} [\cite{Las}, Theorem 1] \label{th103} Let $n=\displaystyle \prod_{i=1}^k p_i^{\alpha_i}$ be the prime factorization of $n \in \mathbb N$ with $\alpha_i > 0$, for all $i$. For an integer $a \geq 1$, the following assertions are equivalent: \\
	i) $a$ is regular (mod $n$); \hspace{1cm}  ii) for every $i \in \{1,2,\dots, k \}$, either $p_i^{\alpha_i}| a$ or  $p_i \nmid a$; \\
	iii) $gcd(a, n)=gcd(a^2, n)$; \hspace{.5cm}  iv) $gcd(a,n)|n$ and $gcd(gcd(a,n),\frac{n}{gcd(a,b)})=1$ ; \\
	v) $a^{\varphi(n)+1}\equiv a ~(mod ~n)$;\hspace{1.2cm}
	vi) there exists an integer $m \geq 1$ such that $a^{m+1}\equiv a~ (mod ~n)$.	   	  	
\end{theorem}
  We denote by $GP(\mathbb Z_n)$, the set of generalized projections (i.e. the set of regular elements) and $P(\mathbb Z_n)=\{a\in \mathbb Z_n~|~a^2=a \}$, the set of projections in $\mathbb Z_n$. 
 Let $R$ be a commutative ring, the relation `$\leq$' defined by: for $a, b \in R$, `$a\leq b$ if and only if either $a=b$ or $a=ab$' is a partial order on $R$ (see \cite{Anil}). In particular, $(\mathbb Z_n,~\leq)$ is a poset with the smallest element $0$ and the largest element $1$. 
    It is known that $(P(\mathbb Z_n),~ \leq)$ is a lattice. However, Khairnar and Waphare \cite{Anil} proved that for any finite commutative ring $R$, $(GP(R),~ \leq)$ is a lattice, hence in particular, $(GP(\mathbb Z_n),~ \leq)$ is a lattice for every $n$.
  Whenever $n$ is a square-free integer, we get that $GP(\mathbb Z_n)=\mathbb Z_n$.   
     In general, $(\mathbb Z_n,~ \leq)$ is not a lattice, for example $(\mathbb Z_9,~ \leq)$ (see Figure \ref{f302}).
  In this paper, we give a necessary and sufficient conditions for the poset ($\mathbb Z_n,~\leq$) to be a lattice.\\ 
\indent   We denote by $U(\mathbb Z_n)$, the set of units in $\mathbb Z_n$ and $N(\mathbb Z_n)$, the set of nilpotents in $\mathbb Z_n$. In the following remark, we list observations required in a sequel.
  \begin{remark} \label{rm202} Let $n=\displaystyle \prod_{i=1}^k p_i^{\alpha_i}$ be the prime factorization of $n \in \mathbb N$ with $\alpha_i > 0$, for all $i$, and let $a\in \mathbb Z_n$. Then,\\
   (i) $a \in U(\mathbb Z_n)$ if and only if $a \equiv $ unit
       (mod $p_i$) for all $i\in \{1,2,\cdots, k\}$.\\
   (ii) $a\in N(\mathbb Z_n)$ if and only if $a \equiv $ zero 
       (mod $p_i$) for all $i\in \{1,2,\cdots, k\}$. \\
  (iii) $a\in GP(\mathbb Z_n)$ if and only if $a \equiv $ zero or
        unit (mod $p_i^{\alpha_i}$) for all $i\in \{1,2,\cdots, k\}$.           	
   \end{remark}
  
   \section{Upper covering projection and lower covering projection}
        In a poset $(P, \leq)$, $a < b$ denotes $a\leq b$ with $a \neq b$.       
   We say that $b$ is an {\it upper cover} of $a$ or $a$ is a {\it lower cover} of $b$ (denoted by $a\prec b$), if $a < b$ and there is no $c\in P$ such that $a < c < b$.  \\
   \indent The following theorem gives an existence of the unique lower cover and the unique upper cover of any element $a\in GP(R)\backslash P(R)$ in the poset $GP(R)$.
   \begin{theorem} [\cite{Anil}, Theorem 2.7] \label{th105} Let $R$ be a finite commutative ring. 
   	If $a \in GP(R)\backslash P(R)$, then there exist a unique $a_u\in GP(R)$ and a unique $a_l\in GP(R)$ such that 
   	$a_l \prec a \prec a_u$. Further, these unique elements are projections and 
   	for $b\in GP(R)$, $a < b$ if and only if $a_u \leq b$; and $b < a$ if and only if $b \leq a_l$.                               
   \end{theorem}
   With notations as in Theorem \ref{th105}, 
   the unique projection $a_l$ is called the {\it lower covering projection} of $a$ 
   and the unique projection $a_u$ is called the {\it upper covering projection} of $a$. If $a$ is a projection,
   then we assume that the lower covering projection of $a$ and the
   upper covering projection of $a$ is $a$ itself.\\
   \indent  Theorem \ref{th105} gives an existence of the upper covering projection and the lower covering projection of any element in $GP(\mathbb Z_n)\backslash P(\mathbb Z_n)$. In this section, we determine the upper covering projection and the lower covering projection of elements in $GP(\mathbb Z_n)\backslash P(\mathbb Z_n)$.\\ 
   \indent 
  The following lemma gives the conditions for strict comparability of a projection and a generalized projection. 
    \begin{lemma} \label{lm301} Let $n=\displaystyle \prod_{i=1}^k p_i^{\alpha_i}$ be the prime factorization of $n \in \mathbb N$ with $\alpha_i > 0$, for all $i$. Let $a\in GP(\mathbb Z_n)$ and $e, f\in P(\mathbb Z_n)\backslash \{1\}$. Then,\\
    	(1) $a < e$ if and only if for $i \in \{1,2,\cdots, k \}$, $e \equiv 0 ~(mod~p_i^{\alpha_i})$ implies that $a \equiv 0 ~(mod~p_i^{\alpha_i})$. \\
    	(2) $f < a$ if and only if for $j \in \{1,2,\cdots, k \}$, $f \nequiv 0 ~(mod~p_j^{\alpha_j})$ implies that $a \equiv 1 ~(mod~p_j^{\alpha_j})$.
   \end{lemma} 
   \begin{proof} $(1)$ Let $a < e$ and $i \in \{1,2,\cdots, k \}$ be such that $e \equiv 0 ~(mod~p_i^{\alpha_i})$. Then $a(1-e)\equiv 0~(mod~n)$ and $1-e \nequiv 0 ~(mod~p_i^{\alpha_i})$. Therefore $a \equiv 0 ~(mod~p_i^{\alpha_i})$. Conversely, suppose that for $i \in \{1,2,\cdots, k \}$, $e \equiv 0 ~(mod~p_i^{\alpha_i})$ implies that $a \equiv 0 ~(mod~p_i^{\alpha_i})$. Then $a (1-e) \equiv 0~(mod~n)$. Thus $a<e$.  \\
   	$(2)$ Let $f < a$ and $j \in \{1,2,\cdots, k \}$ be such that $f \nequiv 0 ~(mod~p_j^{\alpha_j})$. Then $1-a \equiv 0 ~(mod~p_j^{\alpha_j})$. Therefore $a \equiv 1 ~(mod~p_j^{\alpha_j})$. Conversely, suppose that for $j \in \{1,2,\cdots, k \}$, $f \nequiv 0 ~(mod~p_j^{\alpha_j})$ implies that $a \equiv 1 ~(mod~p_j^{\alpha_j})$. This gives $ f (1-a) \equiv 0~(mod~n)$. Thus $f < a$.  	
  \end{proof}	  
   \begin{remark} [\cite{Anil}, Remark 3]\label{rm301} Let $a\in GP(\mathbb Z_n)$. 
   	Suppose $k\geq 2$ be the smallest integer such that $a^k=a$.
   	Then $(a^{k-1})^2=a^{2k-2}=a^ka^{k-2}=aa^{k-2}=a^{k-1}$. Therefore $a^{k-1}\in P(\mathbb Z_n)$, and $a^{k-1}=a_u$. 
   	Clearly, $a \leq a_u$; and $a=a_u$ if and only if $a \in P(\mathbb Z_n)$. 
   \end{remark}
   For any $a\in GP(\mathbb Z_n)\backslash P(\mathbb Z_n)$ the following theorem gives a construction for $a_u$.
    \begin{theorem} Let $n=\displaystyle \prod_{i=1}^k p_i^{\alpha_i}$ be the prime factorization of $n \in \mathbb N$ with $\alpha_i > 0$, for all $i$, and $a\in GP(\mathbb Z_n)\backslash P(\mathbb Z_n)$. If $a\in U(\mathbb Z_n)$, then $a_u=1$. If $a\notin U(\mathbb Z_n)$, then $a_u=b^{\varphi(\frac {n}{b})}$ where $b=\displaystyle \prod_{\substack{j=1 \\ a ~\equiv ~0~(mod~p_j^{\alpha_j}) }}^{k} p_j^{\alpha_j}$.    	
    \end{theorem}
    \begin{proof} Let $a\in GP(\mathbb Z_n)\backslash P(\mathbb Z_n)$. If $a \in U(\mathbb Z_n)$ then by Remark \ref{rm301}, $a_u=1$. Suppose $a\notin U(\mathbb Z_n)$. By Remark \ref{rm202}, there exists $j \in \{1,2, \cdots, k\}$ such that $a\equiv 0~(mod~p_j^{\alpha_j})$, and  $a\equiv 0~or~unit~(mod~p_i^{\alpha_i})$ for all $i\neq j$. Let $b=\displaystyle \prod_{\substack{j=1 \\ a ~\equiv ~0~(mod~p_j^{\alpha_j}) }}^{k} p_j^{\alpha_j}$ and $a_u=b^{\varphi(\frac {n}{b})}$. Then $a_u \in P(\mathbb Z_n)$ and by Lemma \ref{lm301}, $a \leq a_u$. Let $e\in P(\mathbb Z_n)$ be such that $a < e$. If $e=1$, then $a_u \leq e$. Suppose $e \neq 1$. Again by Lemma \ref{lm301}, for $i \in \{1,2,\cdots, k \}$, $e \equiv 0 ~(mod~p_i^{\alpha_i})$ implies that $a \equiv 0 ~(mod~p_i^{\alpha_i})$. Hence, for $i \in \{1,2,\cdots, k \}$, $e \equiv 0 ~(mod~p_i^{\alpha_i})$ implies that $a_u \equiv 0 ~(mod~p_i^{\alpha_i})$. Thus $a_u \leq e$.    	
   \end{proof} 
   For any $a\in GP(\mathbb Z_n)\backslash P(\mathbb Z_n)$ the following theorem gives a construction for $a_l$.	
  \begin{theorem} Let $n=\displaystyle \prod_{i=1}^k p_i^{\alpha_i}$ be the prime factorization of $n \in \mathbb N$ with $\alpha_i > 0$, for all $i$, and $a\in GP(\mathbb Z_n)\backslash P(\mathbb Z_n)$. Then $a_l=b^{\varphi(\frac {n}{b})}$ where $b=\displaystyle \prod_{\substack{j=1 \\ a ~\nequiv ~1~(mod~p_j^{\alpha_j}) }}^{k} p_j^{\alpha_j}$.    	
   	\end{theorem}
   	\begin{proof} Let $a \in GP(\mathbb Z_n)\backslash P(\mathbb Z_n)$. If $a\equiv 1~(mod ~p_j^{\alpha_j})$ for all $j \in \{1,2,\cdots, k\}$ then $a-1\equiv 0~(mod~n)$. Hence $a=1\in P(\mathbb Z_n)$, a contradiction. Therefore there exists $j \in \{1,2,\cdots, k\}$ such that $a \nequiv ~1~(mod~p_j^{\alpha_j})$. Let $b=\displaystyle \prod_{\substack{j=1 \\ a ~\nequiv ~1~(mod~p_j^{\alpha_j}) }}^{k} p_j^{\alpha_j}$ and $a_l=b^{\varphi(\frac {n}{b})}$.   		
   We prove that $a_l \leq a$. Let $i\in \{1,2,\cdots, k\}$ be such that $a_l\nequiv 0~(mod~p_i^{\alpha_i}) $. Then $b\nequiv 0~(mod~p_i^{\alpha_i})$ and hence $a \equiv 1~(mod~p_i^{\alpha_i})$. This yields, $a_l(a-1) \equiv 0~(mod~n)$. Therefore $a_l \leq a$. Let $f \in P(\mathbb Z_n)$ be such that $f < a$, and $j\in \{1,2,\cdots, k\}$ be such that $f \nequiv 0~(mod~p_j^{\alpha_j})$. Then by Lemma \ref{lm301}, we get $a \equiv 1~(mod~p_j^{\alpha_j})$. Consequently, $b \nequiv 0~(mod~p_j^{\alpha_j})$ and hence $a_l \nequiv 0~(mod~p_j^{\alpha_j})$. This implies that $a_l \equiv 1~(mod~p_j^{\alpha_j})$. Therefore $f(1-a_l) \equiv 0~(mod~n) $. Thus $f \leq a_l$.   
   \end{proof}	
    Let $P$ be a poset and $a,b \in P$. The {\it join} of $a$ and $b$, denoted by
     $a\vee b$, is defined as $a \vee b = \sup~ \{a, b\}$. The {\it meet} of $a$ and $b$, denoted by $a\wedge b$, is defined as $a \wedge b= inf~ \{a, b\}$.\\
    \indent  We conclude this section with the following examples.\\
           \begin{figure}[h]
      	\begin{center}
      		\begin{tikzpicture}[scale=.65]     
      		
      		\draw (-6,0)--(-6,2)--(-6,4)--(-6,6);       
      		
      		\node [below] at(-6,0) {$0$}; \draw[fill=white](-6,0) circle(.06);
      		\node [left] at(-6,2) {$2$};  \draw[fill=white](-6,2) circle(.06);
      		\node [left] at(-6,4) {$3$}; \draw[fill=white](-6,4) circle(.06);    
      		\node [above] at(-6,6) {$1$}; \draw[fill=white](-6,6) circle(.06); 
      		
      		\node[below] at (-6,-1) {Lattice of $ \mathbb Z_{4}$};

      		\draw (4,0)--(4,2)--(4,4); 
      		\draw (4,0)--(2,2)--(4,4); 
      		\draw (4,0)--(6,2)--(4,4);  
      		
      		\draw (2,2)--(0,4)--(2,6); 
      		\draw (2,2)--(2,4)--(2,6); 
      		\draw (4,4)--(2,6);     
      		
      		\node [below] at(4,0) {$0$}; \draw[fill=white](4,0) circle(.06);
      		\node [left] at(4,2) {$6$};  \draw[fill=white](4,2) circle(.06);
      		\node [above] at(4,4) {$5$}; \draw[fill=white](4,4) circle(.06);    
      		\node [left] at(2,2) {$4$}; \draw[fill=white](2,2) circle(.06); 
      		\node [right] at(6,2) {$2$}; \draw[fill=white](6,2) circle(.06); 
      		
      		\node [left] at(0,4) {$3$}; \draw[fill=white](0,4) circle(.06); 
      		\node [above] at(2,6) {$1$}; \draw[fill=white](2,6) circle(.06);
      		\node [left] at(2,4) {$7$}; \draw[fill=white](2,4) circle(.06); 
      		
      		\node[below] at (4,-1) {Lattice of $ \mathbb Z_{8}$};
      		\end{tikzpicture}
      	\end{center}
      	
      	\caption{} \label{f301}
      	
      \end{figure}  
     \indent  In the following example, $n \in \mathbb N$ is not square-free but the poset $\mathbb Z_n$ is a lattice.
    \begin{example} \label{ex301} Consider the ring $\mathbb Z_4$. Then $GP(\mathbb Z_4)= \{0,1,3 \}$ and $N(\mathbb Z_4)=\{0,2\} $. Note that $4$ is not square-free but the poset $\mathbb Z_4$ is a lattice (see Figure \ref{f301}). Also, the nilpotent element $2$ possess  unique upper cover.
    \end{example}        
  \begin{example} \label{ex302} Consider the ring $\mathbb Z_8$. Then $GP(\mathbb Z_8)= \{0,1,3,5,7 \}$ and $N(\mathbb Z_8)=\{0,2,4,6\} $.  Note that $8$ is not square-free but the poset $\mathbb Z_8$ is a lattice (see Figure \ref{f301}). Also, each of $2$ and $6$ possess unique upper covers but $4$ does not possess unique upper cover.   	  
  \end{example} 	  
   	
       \begin{figure}[h]
       	\begin{center}
       		\begin{tikzpicture}[scale=.65]     
       		
       		\draw (-4,0)--(-4,4)--(-4,6); 
       		\draw (-4,0)--(-6,4)--(-4,6); 
       		\draw (-4,0)--(-8,4)--(-4,6);  
       		
       		\draw (-4,0)--(-2,2)--(-2,4)--(-4,6); 
       		\draw (-4,0)--(0,2)--(0,4)--(-4,6); 
       		\draw (-2,2)--(0,4);
       		\draw (0,2)--(-2,4);    
       		
       		\node [below] at(-4,0) {$0$}; \draw[fill=white](-4,0) circle(.06);
       		\node [left] at(-4,4) {$8$};  \draw[fill=white](-4,4) circle(.06);
       		\node [above] at(-4,6) {$1$}; \draw[fill=white](-4,6) circle(.06);    
       		\node [left] at(-6,4) {$2$}; \draw[fill=white](-6,4) circle(.06); 
       		\node [left] at(-8,4) {$5$}; \draw[fill=white](-8,4) circle(.06); 
       		
       		\node [left] at(-2,2) {$3$}; \draw[fill=white](-2,2) circle(.06); 
       		\node [left] at(-2,4) {$4$}; \draw[fill=white](-2,4) circle(.06);
       		\node [right] at(0,2) {$6$}; \draw[fill=white](0,2) circle(.06);
       		\node [right] at(0,4) {$7$}; \draw[fill=white](0,4) circle(.06);
       		
       		\node[below] at (-4,-1) {Poset of $ \mathbb Z_{9}$};

       		\draw (6,0)--(6,2)--(6,5)--(6,6); 	
       		\draw (6,0)--(4,2)--(4,3)--(4,4)--(6,5);
       		\draw (6,0)--(8,2)--(6,5);
       		\draw (8,2)--(7,5)--(6,6);
       		\draw (8,2)--(8,3)--(8,4)--(8,5)--(6,6);
       		
       		\node [below] at(6,0) {$0$}; \draw[fill=white](6,0) circle(.06);
       		\node [left] at(6,2) {$2$}; \draw[fill=white](6,2) circle(.06);
       		\node [left] at(6,5) {$7$}; \draw[fill=white](6,5) circle(.06);
       		\node [above] at(6,6) {$1$}; \draw[fill=white](6,6) circle(.06);
       		\node [left] at(4,2) {$8$}; \draw[fill=white](4,2) circle(.06);
       		\node [left] at(4,3) {$4$}; \draw[fill=white](4,3) circle(.06);
       		\node [left] at(4,4) {$10$}; \draw[fill=white](4,4) circle(.06);
       		\node [right] at(8,2) {$6$}; \draw[fill=white](8,2) circle(.06);
       		\node [left] at(7,5) {$11$}; \draw[fill=white](7,5) circle(.06);
       		\node [right] at(8,3) {$3$}; \draw[fill=white](8,3) circle(.06);
       		\node [right] at(8,4) {$9$}; \draw[fill=white](8,4) circle(.06);
       		\node [right] at(8,5) {$5$}; \draw[fill=white](8,5) circle(.06);
       		
       		\node[below] at (6,-1) {Lattice of $ \mathbb Z_{12}$};
       		
       		\end{tikzpicture}
       	\end{center}
       	
       	\caption{} \label{f302}
       	
       \end{figure}    
    \indent In the following example, the poset $\mathbb Z_n$ is not a lattice. Also, none of the nilpotent elements possess unique upper cover.
     \begin{example} \label{ex303} Consider the ring $\mathbb Z_9$. Then $GP(\mathbb Z_9)= \{0,1,2,4,5,7,8 \}$ and $N(\mathbb Z_9)=\{0,3,6\} $. By Figure \ref{f302}, the poset $\mathbb Z_9$ is not a lattice. Note that $3 \vee 6$ and $4 \wedge 7$ do not exist. Also, each of $3$ and $6$ do not possess unique upper covers and each of $4$ and $7$ do not possess unique lower covers.    	  
     \end{example}
     In the following example, $n$ is not square-free but the poset $\mathbb Z_n$ is a lattice.   
     \begin{example} \label{ex304} Consider the ring $\mathbb Z_{12}$. Then $GP(\mathbb Z_{12})= \{0,1,3,4,5,7,8,9,11 \}$ and $N(\mathbb Z_{12})=\{0,6\} $. Note that $12$ is not square-free but the poset $\mathbb Z_{12}$ is a lattice (see Figure \ref{f302}). Observe that, the nilpotent element $6$ does not possess an unique upper cover.    	  
     \end{example} 
    \indent In the next section, we give a necessary and sufficient condition for the existence of supremum and infimum of any two elements of the poset $\mathbb Z_n$.
    \section{Existence of $a\vee b$ and $a\wedge b$ for $a,b \in \mathbb Z_n$} 
    For $x \in \mathbb Z_n$, the ideal generated by $x$ is denoted by $(x)$.\\
  \indent The following theorem characterizes the existence of $a\vee b$ for $a,b \in \mathbb Z_n$.    
   \begin{theorem} \label{th307} Let $n=\displaystyle \prod_{i=1}^k p_i^{\alpha_i}$ be the prime factorization of $n \in \mathbb N$ with $\alpha_i > 0$, for all $i$. Let $a, b \in \mathbb Z_n $ be incomparable and $d=gcd(gcd(a,b),n)$. Then, $a \vee b$ exists if and only if the coset $(\frac{n}{d})+1$ has the smallest element.   	
   \end{theorem}  
   \begin{proof} For each $i \in \{1,2, \cdots, k\}$, let $\beta_i, \gamma_i \in \mathbb W=\mathbb N \bigcup \{0\}$ be the largest powers of prime $p_i$ such that $a \equiv 0~(mod ~p_i^{\beta_i})$ and $b \equiv 0~(mod ~p_i^{\gamma_i})$ respectively. Let $f_i=max\{(\alpha_i-\beta_i), (\alpha_i-\gamma_i), 0\}$ and $m=\displaystyle \prod_{i=1}^k p_i^{f_i}$. 
  Then $m=\frac{n}{d}$.	
   	Let $c \in \mathbb Z_n$ and for each $i \in \{1,2, \cdots, k\}$, let $t_i \in \mathbb N$ be the largest powers of prime $p_i$ such that $c \equiv 1~(mod ~p_i^{t_i})$. Then, $a < c$ and $ b < c$, if and only if $a(c-1)\equiv 0~(mod~n)$ and $b(c-1)\equiv 0~(mod~n)$, if and only if $t_i \geq (\alpha_i-\beta_i), (\alpha_i-\gamma_i)$ for all $i$, if and only if $c-1 \in (\frac{n}{d})$, if and only if $c \in (\frac{n}{d})+1$.\\
   	\indent Suppose $a\vee b$ exists. Since $a$ and $b$ are incomparable, we have $a < a\vee b$ and $b < a \vee b$. This yields $a\vee b \in (\frac{n}{d})+1$. Let $x \in (\frac{n}{d})+1$. Then $a < x $ and $b < x$. Therefore $a\vee b \leq x$. Thus $a\vee b$ is the smallest element of the coset $(\frac{n}{d})+1$. Conversely, suppose that the coset $(\frac{n}{d})+1$ has the smallest element, say $e \in (\frac{n}{d})+1$. This yields $a < e$ and $b < e$. We claim that $a\vee b =e$. Let $f \in \mathbb Z_n$ be such that $a < f$ and $b < f$. Then $f\in (\frac{n}{d})+1$. Therefore $e \leq f$. Thus $a\vee b=e$.   	
   \end{proof} 
   From the proof of Theorem \ref{th307}, it is clear that, if $a\vee b$ exists, then $a \vee b$ is the least element of the coset $(\frac{n}{d})+1$. Also, if the coset $(\frac{n}{d})+1$ has the smallest element $e$, then $a \vee b=e$. \\
   \indent  The following corollary is an immediate consequence of Theorem \ref{th307}.
   \begin{corollary} \label{cr401} Let $n\in \mathbb N,~ n > 1$ and $S=\{d \in \mathbb N~|~ d=gcd(gcd(a,b),n),$ for some incomparable elements $a,b \in \mathbb Z_n \}$. Then, $\mathbb Z_n$ is a lattice if and only if  every coset in $\{ (\frac{n}{d})+1~|~ d \in S\}$ has smallest element.
   \end{corollary}
  \indent  In the following theorem, we characterize the existence of $a\wedge b$ for $a,b \in \mathbb Z_n$. 	 
  \begin{theorem} \label{th308} Let $n=\displaystyle \prod_{i=1}^k p_i^{\alpha_i}$ be the prime factorization of $n \in \mathbb N$ with $\alpha_i > 0$, for all $i$. Let $a, b \in \mathbb Z_n $ be incomparable and $d=gcd(gcd(a-1,b-1),n)$. Then, $a \wedge b$ exists if and only if the ideal $(\frac{n}{d})$ has the largest element.   	
  \end{theorem}  
  \begin{proof} For each $i \in \{1,2, \cdots, k\}$, let $\beta_i, \gamma_i \in \mathbb W$ be the largest powers of prime $p_i$ such that $a \equiv 1~(mod ~p_i^{\beta_i})$ and $b \equiv 1~(mod ~p_i^{\gamma_i})$ respectively. Let $f_i=max\{(\alpha_i-\beta_i), (\alpha_i-\gamma_i), 0\}$ and $m=\displaystyle \prod_{i=1}^k p_i^{f_i}$.
  Then $m=\frac{n}{d}$.
  Let $c \in \mathbb Z_n$ and for each $i \in \{1,2, \cdots, k\}$, let $s_i \in \mathbb N$ be the largest powers of prime $p_i$ such that $c \equiv 0~(mod ~p_i^{s_i})$. Then, $c < a$ and $ c < b$ if and only if $c(a-1)\equiv 0~(mod~n)$ and $c(b-1)\equiv 0~(mod~n)$ if and only if $s_i \geq (\alpha_i-\beta_i), (\alpha_i-\gamma_i)$ for all $i$ if and only if $c \in (\frac{n}{d})$.\\
  \indent Suppose $a\wedge b$ exists. Since $a$ and $b$ are incomparable, we have $ a\wedge b < a$ and $a \wedge b < b$. This yields $a\wedge b \in (\frac{n}{d})$. Let $x \in (\frac{n}{d})$. Then $x < a $ and $x < b$. Therefore $x < a\wedge b$. Thus $a\wedge b$ is the largest element of the ideal $(\frac{n}{d})$. Conversely, suppose that the ideal $(\frac{n}{d})$ has the largest element, say $e \in (\frac{n}{d})$. This yields $e < a$ and $e < b$. We claim that $a\wedge b =e$. Let $f \in \mathbb Z_n$ be such that $f < a$ and $f < b$. Then $f\in (\frac{n}{d})$. Therefore $f\leq e$. Thus $a\wedge b=e$.  	
  \end{proof} 
   From the proof of Theorem \ref{th308}, it is clear that if $a\wedge b$ exists, then $a \wedge b$ is the largest element of the ideal $(\frac{n}{d})$. Also, if the ideal $(\frac{n}{d})$ has the largest element $e$, then $a \wedge b=e$. \\
 \indent  The following corollary is an immediate consequence of Theorem \ref{th308}. 
   \begin{corollary} \label{cr402} Let $n\in \mathbb N,~ n > 1$ and $S'=\{d \in \mathbb N~|~ d=gcd(gcd(a-1,b-1),n),$ for some incomparable elements $a,b \in \mathbb Z_n \}$. Then, $\mathbb Z_n$ is a lattice if and only if every ideal in $\{(\frac{n}{d})~|~d\in S' \}$ has largest element.
   \end{corollary}  
   \begin{corollary} \label{cr404} Let $n\in \mathbb N,~ n > 1$, $S=\{d \in \mathbb N~|~ d=gcd(gcd(a,b),n),$ for some incomparable elements $a,b \in \mathbb Z_n \}$ and $S'=\{d \in \mathbb N~|~ d=gcd(gcd(a-1,b-1),n),$ for some incomparable elements $a,b \in \mathbb Z_n \}$. Then, every ideal in $\{ (\frac{n}{d})~|~d \in S'\}$ has largest element if and only if every coset in $\{ (\frac{n}{d})+1~|~d \in S\}$ has  smallest element.
   \end{corollary} 	
   \begin{proof} Follows from Corollaries \ref{cr401} and \ref{cr402}.
   \end{proof} 
   
  The following two lemmas relate the largest element of an ideal with the smallest element of a coset and vice versa. 
   \begin{lemma} \label{nlm302} Let $n=n_1 n_2$ with $n_1 \geq 1, n_2 \geq 3$ and $I=(n_1)$, $J=(n_2)$. Then, the largest element of the ideal $I$ becomes the smallest element of the coset $J+1$.   
   \end{lemma}	
   \begin{proof} Since $|I|=n_2 \geq 3$, we have $n_1 \nequiv -n_1 ~(mod ~n)$. Therefore $n_1$ and $-n_1$ are distinct elements in $I$. Let $e_1n_1\in I$ be the largest element of $I$. Then $x_1n_1 \leq e_1n_1$ for all $x_1\in \mathbb Z$.
   This yields $n_1 \leq e_1n_1$ and $-n_1 \leq e_1n_1$. 
   That is $n_1\equiv e_1n_1~(mod~n)$ or $n_1\equiv n_1e_1n_1~(mod~n)$; and $-n_1\equiv e_1n_1~(mod~n)$ or $-n_1\equiv -n_1e_1n_1~(mod~n)$.
   If $n_1\equiv e_1n_1~(mod~n)$ and $-n_1\equiv e_1n_1~(mod~n)$, then $n_1 \equiv -n_1 ~(mod~n)$, a contradiction to the fact that $n_2\geq 3$. Thus, either $n_1\equiv n_1e_1n_1~(mod~n)$ or $-n_1\equiv -n_1e_1n_1~(mod~n)$. Suppose $n_1\equiv n_1e_1n_1~(mod~n)$. That is $n_1(e_1n_1-1)\equiv 0~(mod~n_1n_2)$. This implies that $e_1n_1-1\equiv 0~(mod~n_2)$, hence $e_1n_1\in J+1$. Similarly, $-n_1\equiv -n_1e_1n_1~(mod~n)$ implies that $e_1n_1\in J+1$. Thus, in any case, $e_1n_1\in J+1$.    
   Let $y_2n_2+1 \in J+1$ be any element. Then $(y_2n_2+1)e_1n_1=y_2n_2e_1n_1+e_1n_1 \equiv e_1n_1~(mod~ n)$. Thus $e_1n_1$ is the smallest element of $J+1$.
\end{proof}
   \begin{lemma} \label{lm303} Let $n=n_1 n_2$ with $n_1 \geq 3, n_2 \geq 1$ and $I=(n_1)$, $J=(n_2)$. Then the smallest element of the coset $J+1$ becomes the largest element of the ideal $I$.  
   \end{lemma}    	
    \begin{proof} Since $|J|=n_1 \geq 3$, we have $n_2 \nequiv -n_2 ~(mod ~n)$. Therefore $n_2+1$ and $-n_2+1$ are distinct elements in the coset $J+1$. Let $e_2n_2+1\in J+1$ be the smallest element of $J+1$. Then $e_2n_2+1\leq x_2n_2+1$ for all $x_2\in \mathbb Z$.
    	This yields $e_2n_2+1\leq n_2+1$ and $e_2n_2+1\leq -n_2+1$. 
    	That is $e_2n_2+1\equiv n_2+1~(mod~n)$ or $e_2n_2+1\equiv (e_2n_2+1) (n_2+1)~(mod~n)$; and $e_2n_2+1\equiv -n_2+1~(mod~n)$ or $e_2n_2+1\equiv (e_2n_2+1)(-n_2+1)~(mod~n)$.    	
    	If $e_2n_2+1\equiv n_2+1~(mod~n)$ and $e_2n_2+1\equiv -n_2+1~(mod~n)$, then $n_2+1 \equiv -n_2+1 ~(mod~n)$, a contradiction to the fact that $n_1\geq 3$. Thus, either $e_2n_2+1\equiv (e_2n_2+1) (n_2+1)~(mod~n)$ or $e_2n_2+1\equiv (e_2n_2+1)(-n_2+1)~(mod~n)$.     	
    	Suppose $e_2n_2+1\equiv (e_2n_2+1) (n_2+1)~(mod~n)$. That is $(e_2n_2+1)(n_2)\equiv 0~(mod~n_1n_2)$. This implies that $e_2n_2+1\equiv 0~(mod~n_1)$, hence $e_2n_2+1\in I$. Similarly,  $e_2n_2+1\equiv (e_2n_2+1)(-n_2+1)~(mod~n)$ implies that $e_2n_2+1\in I$. Thus, in any case, $e_2n_2+1\in I$.    
    	Let $y_1n_1 \in I$ be any element. Then $(y_1n_1)(e_2n_2+1)=y_1n_1e_2n_2+y_1n_1 \equiv y_1n_1~(mod~ n)$. Thus $e_2n_2+1$ is the largest element of $I$.    	
   \end{proof}  
   \begin{remark} \label{rm303} Let $a\in GP(\mathbb Z_n)$ and $I$ be the ideal generated by $a$. Then $a_u$ is the largest element of $I$. For $ma \in I$, $maa_u=ma$, hence $ma \leq a_u$. 
   \end{remark}
   If $a,b \in GP(\mathbb Z_n)$ then $a\vee b$ and $a\wedge b$ both exists in the poset $GP(\mathbb Z_n)$ (see \cite{Anil}). The following two theorems gives the existence of  $a\vee b$ and $a\wedge b$ in the poset $\mathbb Z_n$ where $a,b \in GP(\mathbb Z_n)$. 
   \begin{theorem} \label{th309} If $a, b \in GP(\mathbb Z_n)$ then $a \vee b$ exists in the poset $\mathbb Z_n$. Further, $a\vee b \in GP(\mathbb Z_n)$. 	
   \end{theorem} 
   \begin{proof} If $a$ and $b$ are comparable then clearly $a \vee b$ exists and $a \vee b \in GP(\mathbb Z_n)$. Suppose $a$ and  $b$ are incomparable. Let $d=gcd(gcd(a,b),n)$, $I$ be the ideal generated by $d$ and $J$ be the ideal generated by $\frac{n}{d}$. As, $a,b \in GP(\mathbb Z_n)$, by Remark \ref{rm202}(iii), $gcd(a,b)\in GP(\mathbb Z_n)$. Therefore   	
   	 $d\in GP(\mathbb Z_n)$. By Remark \ref{rm303}, the ideal $I$ possesses the largest element, say $e$ and $e\in GP(\mathbb Z_n)$. Since $d|a$ and $d|b$, we have $a,b \in I$. As, $a$ and $b$ are incomparable, we have $|I|=\frac{n}{d} \geq 3$. By Lemma \ref{nlm302}, $e$ becomes the smallest element of the coset $J+1$. By Theorem \ref{th307}, $a\vee b$ exists and $a\vee b= e$. Thus, $a\vee b =e \in GP(\mathbb Z_n)$.  
   \end{proof}
    \begin{theorem} \label{th3010} Let $a, b \in GP(\mathbb Z_n)$ be such that $a-1, b-1 \in GP(\mathbb Z_n)$. Then $a \wedge b$ exists in the poset $\mathbb Z_n$ and $a \wedge b \in GP(\mathbb Z_n)$. 	
    \end{theorem} 
    \begin{proof} If $a$ and $b$ are comparable then clearly $a \wedge b$ exists and $a\wedge b \in GP(\mathbb Z_n)$. Suppose $a$ and  $b$ are incomparable. Let $d=gcd(gcd(a-1,b-1),n)$ and $I$ be the ideal generated by $\frac{n}{d}$.  As, $a-1, b-1 \in GP(\mathbb Z_n)$, by Remark \ref{rm202}(iii), $gcd(a-1, b-1)\in GP(\mathbb Z_n)$. Therefore $d\in GP(\mathbb Z_n)$, and hence  $\frac{n}{d}\in GP(\mathbb Z_n) $. By Remark \ref{rm303}, the ideal $I$ possesses the largest element, say $e$ and $e\in GP(\mathbb Z_n)$. By Theorem \ref{th308}, $a\wedge b$ exists and $a\wedge b=e$. Thus $a\wedge b=e \in GP(\mathbb Z_n)$. 
    \end{proof}
    \indent In the following theorem, we give a necessary and sufficient condition for the poset $\mathbb Z_n$ to be a lattice.  	   
 \begin{theorem} Let $n=\displaystyle \prod_{i=1}^k p_i^{\alpha_i}$ be the prime factorization of $n \in \mathbb N$ with $\alpha_i > 0$, for all $i$. Then,  $\mathbb Z_n$ is  a lattice if and only if for every $n_1 \geq 3$ with $n=n_1 n_2$, $(n_1)$ possess the largest element.   	
 \end{theorem} 
  \begin{proof} Suppose $\mathbb Z_n$ is  a lattice. Let $n=n_1 n_2$ with $n_1 \geq 3$. If $(n_1)$ does not possess the largest element, then $|(n_1)|=n_2 \geq 3$. 
  	Let $a=n_1$; and $b=p n_1$, where $p$ is a prime such that $gcd(p, n_2)=1$ and $n_2\nmid (p-1)$. Then $a,b \nequiv 0~(mod~n)$ and $a \nequiv b~(mod~n)$. If $a < b$, then $a \equiv a b~(mod~n)$. That is $n_1\equiv n_1 p n_1~(mod~n)$. This yields $n_1(pn_1-1) \equiv 0~(mod~n)$. This implies that $pn_1 \equiv 1~(mod~n_2)$. Hence $gcd(n_1, n_2)=1$. Consequently $n_1\in GP(\mathbb Z_n)$. By Remark \ref{rm303}, $(n_1)$ possess the largest element, a contradiction. Therefore $a \nleq b$. Similarly, $b \nleq a$. Thus $a$ and $b$ are incomparable. Observe that $gcd(gcd(a,b),n)=n_1$. As, $(\frac{n}{n_2})=(n_1)$ does not possess the largest element.
  	By Lemma \ref{lm303}, the coset $(n_2)+1$ does not possess the smallest element.	
  	 By Corollary \ref{cr401}, $\mathbb Z_n$ is not a lattice, a contradiction. Therefore $(n_1)$ possess the largest element. Conversely, suppose for every $n_1 \geq 3$ with $n=n_1 n_2$, $(n_1)$ possess the largest element. If $\mathbb Z_n$ is not a lattice, then by Corollary \ref{cr401}, there exists incomparable elements $a', b' \in \mathbb Z_n$ such that $d'=gcd(gcd(a',b'), n)$ and $(\frac{n}{d'})+1$ does not possess the smallest element.
  	Therefore $ \frac{n}{d'}=|(d')|\geq 3$ and $|(\frac{n}{d'})+1|\geq 3$. Hence $|(\frac{n}{d'})|=d' \geq 3$.
  	 Let $n_1'=d'$ and $n_2'=\frac{n}{d'}$. Then $n_1' \geq 3$ and $n=n_1'n_2'$. By Lemma \ref{nlm302}, $(n_1')$ does not possess the largest element, a contradiction. Thus $\mathbb Z_n$ is a lattice.	   
  \end{proof}
  \begin{remark} \label{rm302} Let $a \in N(\mathbb Z_n)\backslash \{0\}$.  If $b \in \mathbb Z_n$ be such that $b < a$ then $b=ba=ba^m$ for any $m \in \mathbb N$. Since $a \in N(\mathbb Z_n)$, we have $b=0$. Hence $a_l=0$. From this, it follows that, if $I$ is an ideal generated by a nilpotent element of $\mathbb Z_n$ such that $|I|\geq 3$, then $I$ does not possess the largest element. Thus $\mathbb Z_n$ is not a lattice.
  \end{remark}	

$$\diamondsuit\diamondsuit\diamondsuit$$

\end{document}